\documentclass[a4paper,12pt]{article}

\usepackage{amsmath}
\usepackage{amssymb}
\usepackage{amsthm}
\usepackage{mathrsfs}
\usepackage{url}

\newtheorem{theorem}{Theorem}[section]
\newtheorem{lemma}[theorem]{Lemma}

\newtheorem{remark}[theorem]{Remark}
\newtheorem{corollary}[theorem]{Corollary}
\newcommand{\proba}{{\mathbb P}}
\newcommand{\eps}{\varepsilon}
\newcommand{\expectation}{{\mathbb E}}
\newcommand{\Oun}{\mathcal{O}(1)}

\newcommand{\V}{\mathcal{V}}

\font\gfont=cmmi10 scaled \magstep{2}     
\newcommand{\gdelta}{\hbox{\gfont \char14}}

\def\a0{4-\sqrt{15}}

\def\R{{\mathbb R}}

\def\N{{\mathbb N}}

\def\F{{\mathcal F}}

\def\No{{\mathscr N}}

\def\L{{\mathscr L}}

\def\lip{\textup{Lip}}
\def\cov{\textup{Cov}}

\def\om{\omega}
\def\eps{\epsilon}

\def\y{{\tilde y}}
\def\z{{\tilde z}}

\def\c{{\mathbf c}}

\begin{document}

\title{{\bf A concentration inequality\\
for interval maps\\with an indifferent fixed point}}

\bigskip

\author{{\bf J.-R. Chazottes}, {\bf P. Collet}\\
{\small Centre de Physique Th\'eorique}\\
{\small CNRS, Ecole Polytechnique}\\
{\small 91128 Palaiseau Cedex, France}\\
{\small e-mail: chazottes@cpht.polytechnique.fr},\\
{\small collet@cpht.polytechnique.fr}\\
{\bf F. Redig}\\
{\small Mathematisch Instituut Universiteit Leiden}\\
{\small Niels Bohrweg 1, 2333 CA Leiden, The Netherlands}\\
{\small e-mail: redig@math.leidenuniv.nl}\\
{\bf E. Verbitskiy}\\
{\small Philips Research, HTC 36 (M/S 2),  5656 AE Eindhoven The Netherlands}\\
{\small e-mail: evgeny.verbitskiy@philips.com}
}

\maketitle

\begin{abstract}
For a map of the unit interval with an indifferent fixed point,
we prove an upper bound for the variance of all observables of $n$ variables $K:[0,1]^n\to\R$ which are componentwise
Lipschitz. The proof is based on coupling and decay of correlation properties of the map.
We then give various applications of this inequality to the
almost-sure central limit theorem, the kernel density estimation, the
empirical measure and the periodogram.

\bigskip

\noindent {\bf key-words}: variance, componentwise Lipschitz
observable, almost-sure central limit theorem, kernel density
estimation, empirical measure, periodogram, shadowing, Kantorovich-Rubinstein theorem.
\end{abstract}

\tableofcontents

\section{Introduction}

Nowadays, concentration inequalities are a fundamental tool in 
probability theory and statistics. We refer the reader
to, {\em e.g.}, \cite{ledoux,dl,massart,mcd,talagrand}. In particular,  
they also turn out to be essential tools to develop a non-asymptotic theory in
statistics, exactly as the central limit theorem and large deviations
are known to play a central part in the asymptotic theory. Besides the
non-asymptotic aspect of concentration inequalities, the crucial
point is that they allow in principle to study random variables
$Z_n=K(X_1,\ldots,X_n)$ that ``smoothly'' depend on the underlying
random variables $X_i$, but otherwise can be defined in an
indirect or a complicated way, and for which explicit computations can
be very hard, even in the case where the $X_i$'s are independent. 

In the context of dynamical systems, central limit theorems and their
refinements, large deviations, and other type of limit theorems have
been proved, almost exclusively for Birkhoff sums of sufficiently
``smooth'' observables. But many natural observables are not Birkhoff sums.
Let us just mention a typical example (see below for more examples), namely the
so-called power spectrum, that is, the Fourier transform of the
correlation function, whose estimator is the integral of the
periodogram. This is a very complicated quantity from the analytic
point of view. Besides the computational difficulties proper to
each observable, one would like to have a systematic method to approach the
questions of fluctuations of observables, instead of designing a
particular method for each case. 

A possible method is concentration inequalities. An
additional difficulty comes in for dynamical systems, namely
the fact that we loose independence, except in very special cases, and
that the mixing properties of dynamical systems are not as nice as for
stochastic processes encountered usually in probability theory, such as
Markov chains, renewal processes, {\em etc}.
So new approaches have to be proposed, based on typical tools of dynamical
systems like the spectral gap (when it exists) of the transfer operator  and the
decay of correlations. The first concentration inequality in this
context was obtained by Collet {\em et al.} \cite{cms} for uniformly
expanding maps of the interval, without assuming the existence of a
Markov partition. They obtained the so-called Gaussian concentration
inequality (also called exponential concentration inequality) by
bounding the exponential moment of any observable of $n$ variables
only assuming that it is componentwise Lipschitz. They deduced several
applications (kernel density estimation, shadowing, etc). In the hope of
proving concentration inequalities for more general dynamical systems,
one can start with an inequality for the variance, leading to a polynomial concentration inequality.
This was indeed done in \cite{encorenous} for a large class of non-uniformly hyperbolic
systems modeled by a ``Young tower with exponential return times''
\cite{youngexp}.
In \cite{nous}, the authors of \cite{encorenous} showed the usefulness of this variance
inequality (therein called ``Devroye inequality'') through various examples.
Let us also mention another approach based on coupling \cite{cckr,collet}
that gives, {\em e.g.}, an altenative proof of the Gaussian deviation inequality
in the case of uniformly expanding maps of the interval, and also used in the context
of Gibbs random fields.

Regarding Birkhoff sums of ``smooth'' observables ({\em e.g.}, H\"older),
central limit theorems and large deviation estimates have been proved both for
systems modeled by a Young towers with exponential return-time tail
mentioned above and those with a summable return-time
tail, see, {\em e.g.}, \cite{youngexp,youngpoly,melbourne,luc}.
So, a natural question is to try to prove an inequality for the variance
of any observable of $n$ variables only assuming it is componentwise
Lipschitz, as in \cite{encorenous}, but relaxing the exponential decay
of the return-time tail of Young towers \cite{youngpoly}. This would
give a way to analyze fluctuations of complicated observables, which
are not Birkhoff sums. 
The simplest and classical example is a map of the unit interval with an indifferent fixed
point.
In this paper, we prove a variance inequality for the map
$T(x)=x+2^\alpha x^{1+\alpha}$ when $x\in[0,1/2[$ and strictly expanding on $[1/2,1]$,
when $\alpha$ is small enough (Theorem \ref{main}). However, the proof verbatim
applies to the class of maps with a unique indifferent fixed point considered in \cite{hu}.
The major difference with the situation in \cite{encorenous,cms} is
that the transfer operator has no spectral gap and that the decay of
correlations is polynomial instead of being exponential. Therefore
we develop a different approach based on decay of correlations.
We need to control the covariance of $C^0$ functions and Lipschitz
functions, which is done by H. Hu \cite{hu}.
An important ingredient is coupling through the Kantorovich-Rubinstein duality theorem.
At present, we are not able to construct explicitely a coupling for the backward process as the one
constructed in \cite{barbour} for uniformly expanding maps of the interval. This explicit coupling
was used in \cite{collet} in order to prove the Gaussian concentration inequality.
After proving the variance inequality, we show various applications of it, namely, to the almost-sure central limit theorem, the kernel density
estimation, the empirical measure, the integrated periodogram and the shadowing. 

The paper is organized as follows. 
Section \ref{defs} contains the necessary informations on the maps,
while Section \ref{main} contains our main result, namely the variance inequality. In Section
\ref{applications} we give various applications of it. Section
\ref{proof} contains the proof of the Devroye inequality.

\section{The map and its properties}\label{defs}

\subsection{The map and the invariant measure}\label{mark}

For the sake of definiteness, we consider the maps $T:[0,1]\circlearrowleft$ such that on $[0,1/2[$
$$
T(x)=x+2^\alpha x^{1+\alpha}
$$
and such that $|T' x|>1$ and $|T"(x)|<\infty$ on $[1/2,1]$. In fact, all what follows is valid
under the assumptions of H. Hu  \cite{hu}.

For $\alpha\in[0,1[$, this map admits an absolutely continuous invariant probability measure
$d\mu(x)=h(x) dx$, where $h(x)\sim x^{-\alpha}$ when $x$ tends to $0$.

We define the sequence of points $x_\ell$ by $x_0=1$, $x_1=1/2$ and for
$\ell\geq 2$ $T(x_\ell)= x_{\ell-1}$ and $x_\ell<1/2$. It is easy to verify
that the sequence of intervals
$$
I_\ell:=]x_{\ell+1}, x_\ell],
$$
for $\ell=0,1,2,\ldots$, is a Markov partition of the interval $]0,1]$.

We have the behavior, see {\em e.g.} \cite{hu},
\begin{equation}\label{AMP}
|I_\ell |\sim \ell^{-\frac{1}{\alpha}-1} ,
\;
x_\ell\sim\ell^{-\frac{1}{\alpha}}.
\end{equation}

\subsection{Decay of correlations}

The covariance or correlation coefficient $\cov_{v,w}(\ell)$ of two $L^2(\mu)$ functions $v,w:[0,1]\to\R$
is defined, as usual, by
\[
\cov_{v,w}(\ell)=\int v\circ T^\ell \ w\ d\mu -\int v \ d\mu \int w\ d\mu.
\]
When $v=w$, we simply write $\cov_{v}$.

Various people established the (optimal) decay of correlations for the map $T$, namely
$\cov_{v,w}(\ell)\approx\ell^{-\frac{1}{\alpha}+1}$. 
In, {\em e.g.}, \cite{youngpoly} this is proved for $u,v$
both being H\"older. As it will turn out, we need the following estimate proved in \cite{hu}.
There exists a constant $C>0$ such that, for all
$v\in C^0$ and $w$ Lipschitz, we have the following decay:
\begin{equation}\label{decay}
\left|\cov_{v,w}(\ell)\right| 
\leq C  \thinspace \| v\|_{C^0}\ \lip(w)\ \gamma_\ell 
\end{equation}
where
\begin{equation}
\gamma_\ell:=\ell^{-\frac{1}{\alpha}+1}
\end{equation}
and where
$$
\lip(w)=\sup_{x\neq x'}\frac{|w(x)-w(x')|}{|x-x'|}\cdot
$$

This follows from \cite{hu}. The fact that 
$C$ {\em does not depend} on $v$, $w$ is the consequence
of Theorem B.1 in \cite{nous}.

\subsection{Central limit theorem}

Using \cite[Proposition 5.2]{hu} and \cite{liv}, we have a central limit theorem for Lipschitz
observables when $0<\alpha<1/2$:
for any $v$ Lipschitz which is not of the form $h-h\circ T$ and such that $\int v d\mu=0$, we have
\begin{equation}\label{clt}
\mu\left(\frac{1}{\sigma_v\sqrt{n}}\sum_{j=0}^{n-1}v\circ T^j 
\leq t \right)\stackrel{n\to\infty}{\longrightarrow}
\frac{1}{\sqrt{2\pi}}\int_{-\infty}^t e^{-\xi^2/2}\thinspace
d\xi,\ \forall t\in\R,
\end{equation}
where
$$
\sigma^2_v=\cov_{v}(0)+2\sum_{\ell=1}^\infty \cov_{v}(\ell)>0.
$$

\section{Variance inequality}\label{main}

Our main theorem is an upper-bound for the variance
of any componentwise Lipschitz function.

We introduce the convenient notations
$$
T_p^{q}(x)= (T^p (x), T^{p+1}(x), \ldots,T^{q}(x))\quad\textup{and}\quad
z_p^q=z_p,z_{p+1},\ldots,z_{q}
$$
for $0\leq p\leq q$. With this notation, if we take a function
$K$ of $n$ variables, we write, {\em e.g.},
$K(z_1^j,z_j,z_{j+1}^n)$ for $K(z_1,z_2,\ldots,z_n)$.

A real-valued function $K$ on $[0,1]^n$ is said to be componentwise
Lipschitz if, for all $1\leq j\leq n$, the following quantities
are finite:
$$
\lip_j(K):=\sup_{z_1,z_2,\ldots, z_j,z_{j+1},\ldots,z_n}\thinspace
\sup_{z_j\neq \hat{z}_j}
\frac{
\left|
K(z_1^{j-1},z_j,z_{j+1}^n)- K(z_1^{j-1},\hat{z}_j,z_{j+1}^n)
\right|
}
{|z_j-\hat{z}_j|}\cdot
$$

Our main theorem reads as follows.

\begin{theorem}\label{devroye}
Let $T$ be the map defined in Section \ref{defs}.
Then, for any $\alpha\in[0,\a0[$, 
there exists $D=D(\alpha)>0$ such that, for any componentwise Lipschitz
function $K:[0,1]^n\to\R$, we have
\begin{equation}\label{devineq}
\int 
\left(K(T_0^{n-1}(x))-\int K(T_0^{n-1}(y))\thinspace d\mu(y)\right)^2 d\mu(x)
\leq D \sum_{j=1}^n (\lip_j(K))^2.
\end{equation}
\end{theorem}
(This inequality is called "Devroye inequality" in \cite{encorenous,nous}.)

\bigskip

An application of Chebychev's inequality immediately yields the
following concentration inequality.

\begin{corollary}
Under the assumptions of Theorem \ref{main}, we have 
$$
\mu\left(x\in[0,1]:
\left|K(T_0^{n-1}(x))-\int K(T_0^{n-1}(y))\thinspace d\mu(y)\right|\geq t \right)
\leq 
\frac{D \sum_{j=1}^n \lip_j(K)^2}{t^2}
$$
for all $t>0$.
\end{corollary}

\bigskip

\begin{remark}
In our context, we cannot expect a Gaussian concentration bound.
This would give a Gaussian concentration inequality incompatible with large deviation
lower bounds obtained in \cite{melbourne} where, for a large class of H\"older observables $v$,
it is proved that for $\epsilon>0$ small enough
$$
\mu\left(\left\{x\in[0,1]:\left|\frac{S_n v(x)}{n}-\int vd\mu\right|>\epsilon\right\}\right)\geq n^{-(\frac{1}{\alpha}-1+\delta)}
$$
for any $\delta>0$ and infinitely many $n$'s,
where $S_n v=v + v\circ T+\cdots +v\circ T^{n-1}$. This type of inequalities was also obtained in \cite{jerome} under different
conditions.
\end{remark}

\section{Some applications}\label{applications}

We now give some applications of the variance inequality \eqref{devineq}. We follow
\cite{nous} where we obtained them in an abstract setting:
therein we assumed that $(X_k)$ was some real-valued, stationary,
ergodic process satisfying \eqref{devineq}, plus eventually an
extra condition on the auto-covariance of Lipschitz observables,
depending on each specific application. 
By \eqref{decay} we have 
$$
\sum_{\ell=1}^\infty \cov_{v}(\ell)\leq C' (\lip(v))^2
$$
where $C'=C \sum_\ell \gamma_\ell<\infty$.
This condition will be sufficient to apply all the results
from \cite{nous} that we will use.

The standing assumption in this section is that $0<\alpha<\a0$, so
that Theorem \ref{main} holds.

\subsection{Almost-sure central limit theorem}

For an observable $v$ such that $\int v d\mu=0$, define the sequence
of weighted empirical (random) 
measures of the normalized Birkhoff sum by 
$$
\mathcal{A}_{n}(x)=\frac{1}{D_n}\sum_{k=1}^{n}\frac{1}{k}
\;\gdelta_{S_{k}v(x)/\sqrt k}
$$
where $D_n=\sum_{k=1}^{n}\frac{1}{k}$. 

We say that the almost-sure central limit theorem holds if 
for $\mu$ almost every $x$, $\mathcal{A}_{n}(x)$ converges weakly to the
Gaussian measure. In fact, we will prove a stronger statement, namely
that the convergence takes place in the Kantorovich distance. 

Let us recall that the
Kantorovich distance between two probability measures $\mu_1$ and
$\mu_2$ on $\R$ is defined by
\begin{equation}\label{kanto}
\kappa(\mu_1,\mu_2)=\sup_{g\in{\mathcal L}}\int g(\xi)\;d\big(\mu_{1}-\mu_{2}\big)(\xi)
\end{equation}
where ${\mathcal L}$ denotes the set of real-valued Lipschitz functions
on ${\mathbb R}$ with Lipschitz constant at most one.

We denote by $\No(0,\sigma^{2}_v)$ the Gaussian measure with
mean zero and variance $\sigma^{2}_v$. 

\begin{theorem}
Let $v$ be a Lipschitz function which is not of the form $h-h\circ T$ and assume
that $\int v d\mu=0$.
Then, for $\mu$ almost every $x$, one has
$$
\lim_{n\to\infty}\kappa\left(\mathcal{A}_{n}(x),
\No(0,\sigma^{2}_v)\right)=0.
$$
\end{theorem}

The theorem is an immediate application of Theorem 8.1 in
\cite{nous} and \eqref{clt}.

Notice that this theorem immediately implies that 
for $\mu$ almost every $x$ $\mathcal{A}_{n}(x)$ converges weakly to
the Gaussian measure. The weak convergence is proved in \cite{jrseb}
by another method (and not only for the present intermittent map).
In \cite{jrpierre}, a speed of convergence in the Kantorovich
distance was obtained for uniformly expanding maps of the interval using
a Gaussian bound.

\subsection{Kernel density estimation}

We consider the sequence of regularized (random) empirical measures $\mathcal{H}_n(x)$
with densities $(h_n)$ defined by
$$
h_n(x;s)= \frac{1}{n a_n} \sum_{j=1}^{n} \psi((s-T^j (x))/a_n)
$$
where $a_n$ is a positive sequence converging to $0$ and such
that $n a_n$ converges to $+\infty$, and 
$\psi$ (the kernel) is a bounded, non-negative,
Lipschitz continuous function with compact support whose integral
equals $1$. We are interested in the convergence in $L^1(ds)$ 
of this empirical density $h_n(x;\cdot)$ to the density $h(\cdot)$
of the invariant measure $d\mu(x)=h(x)dx$.
This is nothing but the distance in total variation between $\mathcal{H}_n(x)$
and $\mu$:
$$
\textup{dist}_{\scriptscriptstyle{\textup{TV}}}(\mathcal{H}_n(x),\mu)=
\int |h_n(x;s)-h(s)|\ ds.
$$

\begin{theorem}
Let $\psi$ and $a_n$ be as just described. 
Then, there exists a constant $C=C(\psi)>0$ 
such that for any integer $n$ and for any
$t>C(a_n^{1-\alpha} +1/(\sqrt{n} a_n^{2}))$, we have 
$$
\mu\left(\left\{x\in[0,1]:
  \textup{dist}_{\scriptscriptstyle{\textup{TV}}}(\mathcal{H}_n(x),\mu)> t \right)\right\} \leq 
\frac{C}{t^2 n a_n^{2}}\cdot
$$
\end{theorem}

This theorem is a direct consequence of Theorem 6.1 in \cite{nous} (with $\tau=1-\alpha$).  

\subsection{Empirical measure}

The empirical measure associated to $x,Tx,\ldots,T^{n-1} x$ is the
random measure on $[0,1]$ defined by
$$
\mathcal{E}_n(x)= \frac{1}{n} \sum_{j=0}^{n-1} \gdelta_{T^j (x)}
$$
where $\delta$ is the Dirac measure. 
From Birkhoff's ergodic theorem, for $\mu$ almost every $x$
this sequence of random measures weakly converges to $\mu$. 
We want to estimate the speed of this convergence with respect
to the Kantorovich distance \eqref{kanto} (now used for probability
measures  on $[0,1]$).
 
\begin{theorem}
There exists a positive constant $C$ such that
for all $t>0$ and $n\geq 1$, we have
$$
\mu\left(\left\{x\in[0,1]:\kappa(\mathcal{E}_n(x),\mu)> t+
    \frac{C}{n^{1/4}}\right\}\right)
\leq \frac{D}{n t^2}\cdot
$$
\end{theorem}

This is an immediate consequence of Theorem 5.2 in \cite{nous}.

\subsection{Integrated periodogram}

Let $v$ be an $L^2(\mu)$ observable and assume, for the sake of simplicity, that
$\int vd\mu=0$. We recall (see, {\em e.g}, \cite{BD}) that 
the raw periodogram (of order $n$) of the process $(v\circ T^k)$ is
the random variable
$$
I_n^v(\omega;x)= \frac1{n}\left| \sum_{j=1}^n
  e^{-ij\omega}\ \left(v(T^j (x))\right)\right|^2\
$$
where $\omega\in[0,2\pi]$.
The spectral distribution function of order $n$ (integral of the raw
periodogram of order $n$) is given by
$$
J_n^v(\omega;x)= \int_0^\omega I_n^v(s;x)\ ds\, .
$$

Let $\hat{C}_v(\omega)$ be the Fourier cosine transform of the
auto-covariance of $v$, namely
$$
\hat{C}_v(\omega)=\sum_{k=0}^\infty \cos(\omega k)\
\cov_v(k+1) \,.
$$
We will denote by $J^v(\omega)$ the  following quantity
$$
J^v(\omega)= \int_0^\omega (2\hat{C}_v(s) - \cov_v(0))\ ds=
\cov_v(0)\ \omega + 2\sum_{k=1}^\infty \frac{\sin(\omega k)}{k}\
\cov_v(k) \,.
$$

\begin{theorem}
Let $v$ be a Lipschitz observable. Then there
exists a positive constant $C=C(v)$ such that for any $n\geq 1$, one has
$$
\int \big(\sup_{\omega\in [0,2\pi]}
  \big|\tilde{J}_n^v(\omega;x)-J^v(\omega) \big| \big)^2 d\mu(x)\leq 
C\thinspace \frac{(1+\log n)^{4/3}}{n^{2/3}} \cdot
$$
\end{theorem}

This theorem is a direct application of Theorem 3.1, and the remark
just after it, in \cite{nous}.

\subsection{Shadowing and mismatch}

Let $A$ be a set of initial conditions with positive measure. If $x\notin A$, we can
ask how well we can approximate the orbit of $x$ by an orbit starting
from an initial condition in $A$. 

We can measure the average quality of ``shadowing'' by the following 
quantity:
$$
\mathcal{Z}_A(x)= \frac{1}{n} 
\inf_{y\in A} \sum_{j=0}^{n-1} |T^j (x) - T^j (y)|.
$$

\begin{theorem}\label{s1}
Let $A$ be a subset of positive measure. Then,
for all $n\geq 1$, for all $t>0$, one has
$$
\mu\left(\left\{x\in[0,1]: 
\mathcal{Z}_A(x) \geq \frac{1}{n^{\frac13}} \left(t +
    \frac{2^{\frac43} D^{\frac13}}{\mu(A)}\right)\right\}\right) 
\leq \frac{D}{n^{\frac13} t^2}\cdot
$$
\end{theorem}

We can also look at the number of mismatch at a given precision:
for $\epsilon>0$, let
$$
\mathcal{Z}'_{A,\epsilon}(x)= \frac{1}{n} 
\inf_{y\in A} \textup{Card}\{ 0\leq
j\leq n-1 \,:\, |T^j (x)- T^j (y) | >\epsilon\}.  
$$

\begin{theorem}\label{s2}
Let $A$ be a subset of positive measure. Then,
for all $n\geq 1$, for all $t>0$, for any $\epsilon>0$, one has
$$
\mu\left(\left\{x\in[0,1] :\mathcal{Z}'_{A,\epsilon}(x) \geq
\frac{1}{\epsilon^{\frac23} n^{\frac13}} 
\left(t + \frac{2^{\frac43} D^{\frac13}}{\mu(A)} \right)\right\}\right)
\leq \frac{D}{\epsilon^{\frac23} n^{\frac13} t^2}\cdot
$$
\end{theorem}

Theorem \ref{s1} is a direct application of Theorem 7.1
in \cite{nous} whereas Theorem \ref{s2} is a direct application
of Theorem 7.2 in \cite{nous}.
 
\section{Proof of Theorem \ref{devroye}}\label{proof}


\subsection{First telescoping}

Let $(X_n)_{n\in\N_0}$ be the stationary process where $X_0$ is distributed according to $\mu$
and $X_i=T(X_{i-1})$ for $i\geq 1$.  The expectation in this process is denoted by $\expectation$.
We abbreviate $X_i^j:=(X_i,X_{i+1},\ldots,X_j)$ for $i\leq j$.
We denote by $\F_i^n$ the sigma-field generated by $X_i,X_{i+1},\ldots,X_n$ for $i\leq n$ and
by convention $\F_{n+1}^n=\{\emptyset,[0,1]\}$, the trivial sigma-field.
We then have the following telescoping identity (martingale difference decomposition):

\begin{eqnarray*}
K(X_0,\ldots,X_{n-1})-\expectation\left(K(X_0,\ldots,X_{n-1})\right) & = & 
\sum_{i=1}^n \expectation\left( K\big| \F_{i-1}^{n-1}\right)-\expectation\left( K\big| \F_{i}^{n-1}\right) \nonumber \\
& =: & \sum_{i=1}^n \V_i. 
\end{eqnarray*}
The measurable function $\expectation\left( K\big| \F_{i-1}^{n-1}\right)$ is a function of $X_{i-1},\ldots,X_{n-1}$ only.
When evaluated along an orbit segment $T_0^{n-1}(x)$, it takes the value
\begin{eqnarray*}
\expectation\left( K\big| \F_{i-1}^{n-1}\right)(T_0^{n-i}(x))
& = & 
\expectation\left( K(X_0,\ldots,X_{n-1}) \big|  X_{i-1}^{n-1}=T_0^{n-i}(x)\right)\\
& = &
\sum_{y:T^{i-1}(y)=x} \frac{h(y)}{h(x) |(T^{i-1})'(y)|} \thinspace K\left(  T_{0}^{i-2}(y),T_{0}^{n-i}(x)\right).
\end{eqnarray*}
To obtain the second equality, notice that the reversed process $(X_{n-i})_{i=0,\ldots,n}$ is a Markov chain with 
transition probability kernel
\begin{equation}\label{kernel}
\proba(X_0=dz\big| X_1=x)=\sum_{y:T(y)=x} \frac{h(y)}{h(x) |T'(y)|} \thinspace \delta(y-z)
\end{equation}
and similarly
$$
\proba(X_0=dz\big| X_k=x)=\sum_{y:T^k(y)=x} \frac{h(y)}{h(x) |(T^k)'(y)|} \thinspace \delta(y-z).
$$
The identity \eqref{kernel} follows at once from Bayes formula and the identity
$\proba(X_1=x|X_0=z)=\delta(x-T(z))$.
  
Since $\F_{i}^{n-1}\subset \F_{j}^{n-1}$ for $i\geq j$, we have the orthogonality property
$$
\expectation(\V_i \V_j)=0\quad\textup{for}\ i\neq j,
$$
and hence
$$
\expectation\left( K -\expectation(K)\right)^2 = \sum_{i=1}^n \expectation(\V_i^2).
$$  
The function $\V_i$ is $\F_{i-1}^{n-1}$-measurable and
$$
\V_i(T_0^{n-i}(x))=
\expectation\left(K\big|X_{i-1}^{n-1}=T_{0}^{n-i}(x)\right)-\expectation\left(K\big|X_{i}^{n-1}=T_{1}^{n-i}(x)\right).
$$ 
Hence, by Cauchy-Schwarz  inequality, 
\begin{eqnarray*}
\V_i^2(T_0^{n-i}(x)) & \leq &
\int \proba(X_{i-1}=dx'\big| X_{i}^{n-1}=T_1^{n-i}(x))\\
&& \times \left[
\expectation\left(K\big|X_{i-1}^{n-1}=T_{0}^{n-i}(x)\right)-\expectation\left(K\big|X_{i-1}^{n-1}=(x',T_{0}^{n-i}(x))\right)
\right]^2\\
& = & 
\sum_{x':T(x')=T(x)} \frac{h(x')}{h(T(x)) |T'(x')|} \thinspace
\left[
\expectation\left(K\big|X_{i-1}=x\right)-\expectation\left(K\big|X_{i-1}=x'\right)
\right]^2.
\end{eqnarray*}
For $x,x'$ such that $T(x)=T(x')$, let
$$
M_i(x,x'):=\expectation\left(K\big|X_{i}=x\right)-\expectation\left(K\big|X_{i}=x'\right)
$$ 
and $\mu_x^i$ denote the conditional distribution of $X_0,\ldots,X_{i-1}$ given that 
$X_i=x$. By using the Lipschitz property of $K$ one gets
$$
|M_i(x,x')|\leq 2 \lip_i(K)+\left|\int K(z_0^{i-1},0,T_1^{n-i-1}(x))\thinspace (d\mu_x^i(z_0^{i-1})-d\mu_{x'}^i(z_0^{i-1}))\right|
$$
and one obtains
\begin{eqnarray*}
\expectation(\V_{i}^2) & = & \int \sum_{x':T(x')=T(x)} \frac{h(x')}{h(T(x))}\frac{d\mu(x)}{|T'(x')|}\thinspace M_i^2(x,x')\\
& \leq & 
8 (\lip_i(K))^2 + 2\int \sum_{x':T(x')=T(x)} \frac{h(x')}{h(T(x))}\frac{d\mu(x)}{|T'(x')|} \\
&& \times \left[\int K(z_0^{i-1},0,T_1^{n-i-1}(x)) \thinspace (d\mu_x^i(z_0^{i-1})-d\mu_{x'}^i(z_0^{i-1}))\right]^2.
\end{eqnarray*}
Let us further abbreviate
\begin{equation}\label{gaga}
\Gamma_i(x,x'):=
\int K(z_0^{i-1},0,T_1^{n-i-1}(x)) \thinspace (d\mu_x^i(z_0^{i-1})-d\mu_{x'}^i(z_0^{i-1})).
\end{equation}
We then obtain
\begin{eqnarray}\label{gogo}
\expectation\left(K-\expectation K\right)^2 & \leq & 
8 \sum_{i=1}^{n}(\lip_i(K))^2 \nonumber \\
&& + 2\sum_{i=1}^{n}\int d\mu(x)  \sum_{x':T(x')=T(x)} \frac{h(x')}{h(T(x))|T'(x')|} \thinspace \Gamma_i^2(x,x').
\end{eqnarray}
(Observe that $\Gamma_k(x,x)=0$.)  

\subsection{Second telescoping}

Our aim is now to further estimate the quantity $\Gamma_i(x,x')$ by using a second telescoping where the decay of correlations
\eqref{decay} can be used.   
   
Let
$$
\Psi_k(x):=\expectation\left( 
K(X_0,\ldots,X_{k-1},0,T_1^{n-k-1}(x))\big| X_k=x\right).
$$   
With this notation \eqref{gaga} reads
\begin{equation}\label{spouic}
\Gamma_k(x,x')=\Psi_k(x)-\Psi_k(x').
\end{equation}
The idea is now to telescope the $\Psi_k$'s by introducing an independent copy $(Y_i)_{i\in\N_0}$ of the process
$(X_i)_{i\in\N_0}$.   We write
\begin{eqnarray}\label{spouac}
\Psi_k(x) & = &
\sum_{p=1}^k 
\expectation\left[
K(X_0^{p-1},Y_p^{k-1},0,T_1^{n-k-1}(x))-K(X_0^{p-2},Y_{p-1}^{k-1},0,T_1^{n-k-1}(x))
\big| X_k=x\right] \nonumber \\
&& + \expectation\left(K(Y_0,\ldots,Y_{k-1},0,T_1^{n-k-1}(x))\right)
\end{eqnarray}
where now $\expectation$ denotes expectation both with respect to the random variables $X$ and $Y$,
and where we make the convention that, if $Y_i^j$ (resp. $X_i^j$) occurs with $j<i$, then $Y$ (resp. $X$)
is simply not present.   

Combining \eqref{spouic} and \eqref{spouac}, we obtain
\begin{eqnarray*}
\Gamma_k(x,x')=
\sum_{p=1}^k
\int \omega_{p-1}(z_0^{p-1},T_1^{n-k-1}(x))(d\mu_x^{p-1,k-p+1}(z)-d\mu_{x'}^{p-1,k-p+1}(z))
\end{eqnarray*}
where $\mu_x^{p-1,k-p+1}$ is the conditional distribution of $X_0,\ldots,X_{p-1}$ given $X_k=x$, and
where
$$
\omega_{p-1}(z_0^{p-1},T_1^{n-k-1}(x)):=
$$
\begin{equation}\label{bombom}
\expectation\left(K(z_0^{p-1},Y_p^{k-1},0,T_1^{n-k-1}(x))-K(z_0^{p-2},Y_{p-1}^{k-1},0,T_1^{n-k-1}(x))\right),
\end{equation}
where the expectation is taken with respect to $Y$.
Observe that 
$$
\big|\omega_{p-1}(z_0^{p-1},T_1^{n-k-1}(x))\big|\leq \lip_{p}(K).
$$   
We now define the distance
$$
d_{p}(z_0^p,\hat{z}_0^{p}):= \inf \left(2\lip_{p+1}(K),\sum_{j=0}^{p} \lip_{j+1}(K) |z_j-\hat{z}_j|\right)\,.
$$
Without loss of generality, we assume $\inf_j \lip_j(K) >0$. Hence, equipped with the distance $d_p$, $[0,1]^{p+1}$ is a complete, separable, metric
space.
From \eqref{bombom} it follows that
$$
\sup_{z_0^{p-1}\neq \hat{z}_0^{p-1},x_1^{n-k-1}}
\frac{|\om_{p-1}(z_0^{p-1},x_1^{n-k-1})-\om_{p-1}(\hat{z}_0^{p-1},x_1^{n-k-1})|} 
{d_{p-1}(z_0^{p-1},\hat{z}_0^{p-1})}\leq 1,
$$
{\em i.e.}, for each fixed $x$, the function $z_0^{p-1}\mapsto \omega_{p-1}(z_0^{p-1},T_1^{n-k-1}(x))$ is
Lipschitz with respect to the $d_{p-1}$ distance, with Lipschitz norm less than or equal to one.

Denote by $\c^{p,q}_{x,x'}(z_0^{p},\hat{z}_0^{p})$ the Kantorovich-Rubinstein coupling, 
associated with  the distance $d_{p}$,  of the measures $\mu_x^{p,q}$ and $\mu_{x'}^{p,q}$ 
(cf \cite[Theorem 11.8.2, p. 421]{dudley}).

For this coupling we thus have
$$
\int d_p(z_0^p,\hat{z}_0^{p})
\ d\c^{p,k-p}_{x,x'}(z_0^{p},\hat{z}_0^{p})
=
\sup_{f:\textup{Lip}_{d_p}(f)\leq 1} \Big(\int f\ d\mu^{p,k-p}_x- \int f\ d\mu^{p,k-p}_{x'}\Big).
$$
Hence, by the definition of the distance $d_p$  and the 
Kantorovich-Rubinstein duality theorem \cite{dudley}, one gets
\begin{eqnarray}\label{gugus}
|\Gamma_k(x,x')| & = & \left|
\sum_{p=0}^{k-1} 
\int 
\om_p(z_0^{p},T_1^{n-k-1}(x))\ (d\mu_x^{p,k-p}(z_0^{p})-d\mu_{x'}^{p,k-p}(z_0^{p}))\right| \nonumber \\
& = & 
\left|
\sum_{p=0}^{k-1} 
\int 
\left[\om_p(z_0^{p},T_1^{n-k-1}(x))-\om_p(\hat{z}_0^{p},T_1^{n-k-1}(x))\right]
\ d\c^{p,k-p}_{x,x'}(z_0^{p},\hat{z}_0^{p})\right| \nonumber \\
& \leq &
\sum_{p=0}^{k-1} 
\int d_p(z_0^p,\hat{z}_0^{p})
\ d\c^{p,k-p}_{x,x'}(z_0^{p},\hat{z}_0^{p}) \nonumber \\
& = &
\sum_{p=0}^{k-1} \left(
\sup_{f:\textup{Lip}_{d_p}(f)\leq 1} \Big(\int f\ d\mu^{p,k-p}_x- \int f\ d\mu^{p,k-p}_{x'}\Big)
\right)\,.
\end{eqnarray}

In order to estimate $\Gamma_k$, we will now exploit the fact that for $k-p$ ''large" the measure $\mu_x^{p,k-p}$
is ``close" to the invariant measure $\mu$.
More precisely,  passing from $\mu_x^{p,0}$ to $\mu_x^{p,k-p}$ involves $k-p$ iterations of the
normalised Perron-Frobenius operator.

\subsection{Distortion and correlation estimates}

We now proceed by estimating the final expression in \eqref{gugus}.

Define, as usual, the normalised Perron-Frobenius operator
\begin{eqnarray*}
\L w(x) & = & \expectation\left( w(X_0)\big| X_1=x\right) \\
& = & \int w(y) \proba(X_0=dy\big|X_1=x)\\
& = & \sum_{u:T(u)=x} \frac{h(u)}{h(x) |T'(u)|}\ w(u).
\end{eqnarray*}
By the Markov property of the reversed process we have
$$
\L^k w(x)= \expectation\left( w(X_0)\big| X_k=x\right)= \int w(y) \proba(X_0=dy\big|X_k=x).
$$
For $f$ a function of $(p+1)$ variables, define
\begin{eqnarray*}
f_p(x) & = & \expectation\left( f(X_0,\ldots,X_p)\big| X_p=x\right)\\
&= & \sum_{u:T^p(u)=x} \frac{h(u)}{h(x) |(T^{p})'(u)|}\ f(T_0^p(u)).
\end{eqnarray*}
We then have
\begin{eqnarray*}
\int f \ d\mu^{p,k-p}_x & = &
\expectation\left( f(X_0,\ldots,X_p)\big| X_k=x\right)\\
& = & 
\int \expectation\left( f(X_0,\ldots,X_p)\big| X_p=y\right)\thinspace \proba(X_p=dy\big| X_k=x)\\
& =  & \L^{k-p} f_p(x).
\end{eqnarray*}

The next three lemmas will be useful.

\begin{lemma}\label{projo}
Let $f$ be such that $\textup{Lip}_{d_p}(f)\leq 1$. Then, for any $y,\y\in I_\ell$ and
any $m\geq 0$, we have
\begin{equation}\label{louche}
|(\L^m f_p)(y)-(\L^m f_p)(\y)|
\leq  C_p  \ \frac{|y-\y|}{y}
\end{equation}
where
$$
C_p= \Oun \sum_{j=0}^{p} \frac{\lip_{j+1}(K)}{(p-j+1)^{1/\alpha}}\cdot
$$
\end{lemma}

\begin{proof}
Observe that it is enough to prove the lemma in the case where $f$ vanishes at some point. The general
case follows by adding a constant. Without loss of generality, we can assume that 
$\L^m f_p)(y)\leq \L^m f_p)(\y)$. Indeed, the opposite case would lead to the same estimate because
there exists a constant $C>0$ such that $y/\y\leq C$, for all $y,\y\in I_\ell$ and all $\ell$.

Since $f$ vanishes at some point and $\lip_{d_p}(f)\leq 1$, we have
$|f_p(T_0^p(\cdot))|/\lip_{p+1}(K)\leq 2$. We also have 
$|\L^m f_p(T_0^p(\cdot))|/\lip_{p+1}(K)\leq 2$.
Now we use the inequality
$$
1 + \frac{3(a-b)}{5} \leq \frac{1+a}{1+b}
$$
for all $a,b$ such that $-2/3\leq b \leq a\leq 2/3$. Therefore, 
\begin{equation}\label{truci}
\left|\frac{(\L^m f_p)(y)}{3 \lip_{p+1}(K)}-\frac{(\L^m f_p)(\y)}{3 \lip_{p+1}(K)}\right|\leq  \frac{5}{3}
\left(
\frac{(\L^m f_p)(y)+ 3 \lip_{p+1}(K)}
{(\L^m f_p)(\y)+ 3 \lip_{p+1}(K)}-1
\right)\cdot
\end{equation}
We have 
\begin{equation}\label{trucmuch}
\frac{(\L^m f_p)(y)+ 3 \lip_{p+1}(K)}
{(\L^m f_p)(\y)+ 3 \lip_{p+1}(K)}\leq
\frac{h(\y)}{h(y)}\ 
\sup_{z,\z}
\frac{h(z)}{h(\z)} \ \frac{T^{(p+m)'}(\z)}{T^{(p+m)'}(z)}\
\frac{f(T_0^p(z))+3 \lip_{p+1}(K)}{f(T_0^p(\z))+3 \lip_{p+1}(K)} 
\end{equation}
where the supremum is taken over the pairs $(z,\z)$ of pre-images of
$y$ and $\y$ whose iterates lie in the same atoms 
of the Markov partition until $p+m$.
To estimate this, we use the bounds
$$
\frac{h(\y)}{h(y)} \leq 1 + C \frac{|y-\y|}{y},\quad
\frac{T^{(p+m)'}(\z)}{T^{(p+m)'}(z)}\leq 1 + C \frac{|y-\y|}{y}
$$
proved in \cite{hu}: the first one follows from the fact that  $h$ belongs to the space $\mathcal{G}$ \cite[p. 502]{hu}
whereas the second one is \cite[Proposition 2.3 (ii) ]{hu}.
We also use the bounds 
\begin{equation}\label{ping}
\frac{|z-\z|}{z} \leq C \frac{|y-\y|}{y}\, ,
\end{equation}
and
\begin{equation}\label{pong}
| T^j(z) - T^j(\z)| \leq \frac{C}{(p-j+1)^{1/\alpha}} \ 
\frac{ |T^p(z) - T^p(\z)|}{T^p(z)}
\end{equation}
which are proved in the appendix (Lemmas \ref{ineq1} and \ref{ineq2}).

Therefore, using \eqref{pong}, we get
\begin{eqnarray*}
\frac{f(T_0^p(z))+3 \lip_{p}(K)}{f(T_0^p(\z))+3 \lip_{p+1}(K)} & \leq  &
1+ \Oun \frac{|f(T_0^p(z))-f(T_0^p(\z))|}{\lip_{p+1}(K)}\\
& \leq &  1+ \frac{\Oun}{\lip_{p+1}(K)} \sum_{j=0}^{p} \frac{\lip_{j+1}(K)}{(p-j+1)^{1/\alpha}}\frac{ |T^p(z) - T^p(\z)|}{T^p(z)}\cdot
\end{eqnarray*}
Using \eqref{ping}  and all the previous bounds  in \eqref{trucmuch}, we obtain
$$
\frac{(\L^m f_p)(y)+ 3 \lip_{p+1}(K)}
{(\L^m f_p)(\y)+ 3 \lip_{p+1}(K)}\leq
1 + \frac{\Oun}{\lip_{p+1}(K)}  \frac{|y-\y|}{y} \sum_{j=0}^{p} \frac{\lip_{j+1}(K)}{(p-j+1)^{1/\alpha}}\cdot
$$
This inequality together with \eqref{truci} completes
the proof of the lemma.
\end{proof}

\begin{lemma}\label{quiche}
Let $f$ be such that $\textup{Lip}_{d_p}(f)\leq 1$. Then for any $q\geq 0$ we have
$$
\int \left|\L^q\left(f_p - \int f_p\ d\mu\right)\right| \ d\mu \leq
D_p \ \gamma_q^{\frac{1-\alpha}{3}}
$$
where 
$$
D_p=\Oun\ C_p\,.
$$
\end{lemma}

\begin{proof}
Let $M>0$ be an integer and $\eps>0$ to be fixed later on.
Recall  the notation $I_\ell=]x_{\ell+1},x_{\ell}]$. For $\ell\leq M$, we define the sequence of
functions $f_p^\ell$, each vanishing outside $I_\ell$, given by
$$
f_p^\ell(x):=\left\{
\begin{array}{l}
\frac{x-x_{\ell+1}}{\eps |I_\ell|} f_p(x_{\ell+1}+\eps |I_\ell|)
\quad\textup{for}\quad x\in [x_{\ell+1},x_{\ell+1}+\eps |I_\ell|] \\
\\ 
f_p(x) \quad  \textup{for} \quad x\in [x_{\ell+1}+\eps |
I_\ell|,x_{\ell}-\eps |I_\ell|]\\ \\ 
\frac{x_{\ell}-x}{\eps |I_\ell|} f_p(x_{\ell}-\eps | I_\ell |) \quad
\textup{for}\quad  x\in [x_{\ell}-\eps | I_\ell|,x_{\ell}] \,. 
\end{array}\right.
$$ 
We have the identity
$$
\L^q \left(f_p -\int f_p \ d\mu\right) = 
$$
$$
\sum_{\ell=0}^M \L^q \left(f_p^\ell -\int f_p^\ell \ d\mu\right) + \L^q \left(
f_p - \sum_{\ell=0}^M f_p^\ell\right) 
- \L^q \left(\int f_p \ d\mu- \sum_{\ell=0}^M \int f_p^\ell \ d\mu\right) \,.
$$
The decay of correlations \eqref{decay} gives us
\begin{eqnarray*}
\int \left| \L^q \left(f_p^\ell - \int f_p^\ell \ d\mu\right)\right| \ d\mu & =& 
\sup_{u: \|u\|_{C^1}\leq 1} \int u\ \L^q \left(f_p^\ell - \int f_p^\ell \ d\mu\right)  d\mu\\
& \leq &
\frac{C_p}{\eps | I_\ell|}\ \gamma_q,
\end{eqnarray*}
since $|f_p|\leq C_p$ and using Lemma \ref{projo} with $m=0$.

On the other hand, we have 
$$
\int \left| f_p - \sum_{\ell=0}^M f_p^\ell\right|\ d\mu \leq 
2 C_p \ \sum_{\ell=0}^M \eps \ | I_\ell| h_\ell  + C_p \
\sum_{\ell=M+1}^\infty  | I_\ell | h_\ell \,. 
$$
The optimal bound is obtained with 
$$
\eps= \gamma_q^{\frac{1}{2}} M^{1+\frac{1}{2\alpha}}\, , \; M= \gamma_q^{-\frac{\alpha}{3}}\,.
$$
The Lemma follows.
\end{proof}

\begin{lemma}\label{pizza}
Let $f$ be such that $\textup{Lip}_{d_p}(f)\leq 1$. Then for any $q\geq 0$ and $\ell\geq 0$
we have
 $$
 \left\|\L^q\left(f_p(x)-\int f_p\ d\mu\right)\right\|_{L^\infty(I_\ell)}\leq \Delta(\ell,q;f_p)
 $$
 where
$$
 \Delta(\ell,q;f_p):=
 \left\{
 \begin{array}{l}
  \frac{2}{|I_\ell|} \ \int_{I_\ell}  |g_{q,f_p}| dx \quad\textup{if}\quad 
 C_p \ |I_\ell |^2 \leq x_{\ell}  \int_{I_\ell}  |g_{q,f_p}|\ dx
  \\ \\
  2 \sqrt{\frac{C_p \int_{I_\ell}  |g_{q,f_p}| dx}{x_{\ell}}} \quad \textup{otherwise}
 \end{array}
 \right.
 $$
 where
 $$
 g_{q,f_p}=\L^q \left(f_p - \int f_p\ d\mu\right).
 $$
\end{lemma}

\begin{proof}
By \eqref{louche} we have
$$
| g_{q,f_p}(y)-g_{q,f_p}(y')| \leq C_p\ \frac{|y-y'|}{x_{\ell}}
$$
for $y,y'\in I_\ell$.

Hence, if we let $J\subseteq I_\ell$ and $y\in J$, then we have
$$
| g_{q,f_p}(y)| \leq \frac{1}{|J|} \int_J | g_{q,f_p}(y')|\ dy'
+\frac{1}{|J|} \int_J | g_{q,f_p}(y)-g_{q,f_p}(y')|\ dy' 
\leq 
$$
$$
\frac{1}{|J|}  \int_{J}  |g_{q,f_p}| dx+ C_p\ \frac{|J|}{x_{\ell}}\ .
$$
The first case follows by taking $J=I_\ell$. In the second case, we take $J$ such that 
$$
|J| =\sqrt{\frac{x_{\ell}}{C_p} \  \int_{I_\ell}  |g_{q,f_p}| dx}\leq |I_\ell|\,.
$$
The lemma is proved.
\end{proof}

Now return to \eqref{gogo}. We have to estimate
\begin{equation}\label{gogagi}
\int d\mu(x)\  \sum_{x' : x\neq x', T(x)=T(x')}
\frac{h(x')}{h(T(x))|T'(x')|} \Gamma_k^2(x,x')=: S_1(k) + S_2(k) 
\end{equation}
where 
$$
S_1(k):=\sum_{m=1}^\infty 
\int_{I_m} d\mu(x)\  \sum_{x' : x\neq x', T(x)=T(x')}
\frac{h(x')}{h(T(x))|T'(x')|} \Gamma_k^2(x,x') 
$$
and 
$$
S_2(k):=
\int_{I_0} d\mu(x)\  \sum_{x' : x\neq x', T(x)=T(x')}
\frac{h(x')}{h(T(x))|T'(x')|} \Gamma_k^2(x,x'),
$$
where the intervals $I_\ell$ form the Markov partition defined in Subsection \ref{mark}.

\bigskip

We have the following lemmas.

\begin{lemma}\label{crachin}
Let 
$$
Q_k:=\sum_{\ell=1}^\infty |I_\ell| \ \sup_{x\in I_\ell, x'\in I_0:T(x)=T(x')} \Gamma_k^2(x,x')\,.
$$
Then there exists a constant $B>0$ such that for any $k$
$$
S_1(k)\leq B\ Q_k
$$
and
$$
S_2(k)\leq B\ Q_k\,.
$$
\end{lemma}

\begin{proof}
We first observe that, if $m\geq 1$, $x\in I_m$,  $T(x)=T(x')$, and
$x\neq x'$,  then $x'\in I_0$. Next, using \cite[Lemma 4.4 (iv)]{hu}
and the fact that $h$ is bounded on $I_0$, we get
$$
G:=\sup_{m\geq 1}\sup_{x\in I_m} \sup_{x'\in I_0} \frac{h(x')h(x)}{h(T(x))|T'(x')|}<\infty\,.
$$
The bound on $S_1(k)$ follows immediately.

For the bound on $S_2(k)$, we have
$$
S_2(k)=\sum_{m\geq 1} 
\int_{I_0} d\mu(x)\  \chi_{I_m}(x')\ \sum_{x' : x\neq x', T(x)=T(x')}
\frac{h(x')}{h(T(x))} \frac{\Gamma_k^2(x,x')}{|T'(x')|}\cdot
$$
Note that the term corresponding to $m=0$ is absent because $x\neq x'$.
Observe that
$$
G':=\sup_{m\geq 1}\sup_{x'\in I_m} \sup_{x\in I_0, T(x)=T(x')}
\frac{h(x')h(x)}{h(T(x))|T'(x')|}<\infty 
$$ 
and there is a constant $C>0$ such that for any $m\geq 1$ 
$$
\left|\{x\in I_0 | T(x)\in T(I_m)\}\right| \leq C\ |I_m|\,.
$$
The lemma follows.
\end{proof}

\begin{lemma}\label{quasidevroye}
Assume that $\alpha\in[0,\a0[$. Then there exists a constant $H>0$ such that
$$
\sum_{k=1}^n Q_k \leq H\ \sum_{j=1}^n (\lip_j(K))^2\,.
$$
\end{lemma}

\begin{proof}
Observe that
$$
\sup_{x\in I_0, x'\in I_m} |\Gamma_k (x,x')| \leq
$$
$$
\sum_{p=0}^{k-1} \sup_{f:\textup{Lip}_{d_p}(f)\leq 1} \Delta(0,k-p,f_p) 
+
\sum_{p=0}^{k-1} \sup_{f:\textup{Lip}_{d_p}(f)\leq 1} \Delta(m,k-p,f_p)
$$
$$
=\Sigma(0,k)+\Sigma(m,k)
$$
where 
$$
\Sigma(j,k):=\sum_{p=0}^{k-1} \sup_{f:\textup{Lip}_{d_p}(f)\leq 1} \Delta(j,k-p,f_p)
$$
for $j\geq 0$.

By Lemma \ref{quiche}
$$
\sup_{f:\textup{Lip}_{d_p}(f)\leq 1} \int  |g_{q,f_p}| d\mu\leq \Oun\
C_p \ \gamma_q^{(1-\alpha)/3}\,. 
$$
Both cases of Lemma \ref{pizza} lead to the bound
$$
\sup_{f:\textup{Lip}_{d_p}(f)\leq 1} \Delta(0,k-p,f_p)  \leq \Oun\
C_p \ \gamma_{k-p}^{(1-\alpha)/6}\,. 
$$
Since $\alpha\in[0,\a0[$, we have
$$
\sum_{q} \gamma_{q}^{(1-\alpha)/6}<\infty
$$
and Young's inequality yields
$$
\sum_{k=1}^{n} \sum_m |I_m| \ \Sigma(0,k)^2 \leq \Oun \ \sum_p C_p^2 \leq \Oun \ \sum_{j=1}^n (\lip_j(K))^2\,.
$$
We now bound
$$
\sum_{\ell} |I_\ell| \ \Sigma(\ell,k)^2\,.
$$

By Lemma \ref{pizza} we get
$$
\sum_{\ell} |I_\ell| \ \Sigma(\ell,k)^2 \leq A_1(k) + A_2(k)
$$
where
$$
A_1(k):= 8 \sum_{\ell} |I_\ell| \ \left(\sum_{p=0}^{k-1}
  \sup_{f:\textup{Lip}_{d_p}(f)\leq 1}\ \sqrt{\frac{C_p \int_{I_\ell}
      |g_{k-p,f_p}(x)| dx}{x_{\ell}}}\thinspace\right)^2 
$$
and 
$$
A_2(k):=8 \sum_{\ell} |I_\ell| \ \left(\sum_{p=0}^{k-1}
  \frac{1}{|I_\ell|} \ \sup_{f:\textup{Lip}_{d_p}(f)\leq 1} \
  \int_{I_\ell}  |g_{k-p,f_p}(x)| dx \right)^2\,. 
$$  
Observe that
$$
\int_{I_\ell}  |g_{k-p,f_p}(x)| dx \leq \frac{\Oun}{\ell}\ \int_{I_\ell}  |g_{k-p,f_p}| d\mu
\leq \frac{\Oun}{\ell}\ C_p \ \gamma_{k-p}^{\frac{1-\alpha}{3}}
$$
since $h_{|I_\ell}\sim \ell$ and  by using Lemma \ref{quiche}.
Hence,
$$
A_1(k)\leq \Oun\ \sum_\ell \frac{|I_\ell|}{\ell x_{\ell}} \ \left( 
\sum_{p=0}^{k-1} C_p\ \gamma_{k-p}^{\frac{1-\alpha}{6}}
\right)^2
$$
which implies, as above,
$$
\sum_{k=1}^n A_1(k)\leq \Oun \ \sum_{j=1}^n (\lip_j(K))^2\,.
$$
We now bound $A_2(k)$.  By Cauchy-Schwarz inequality, for any $\delta>0$, we have
$$
A_2(k)\leq \Oun \ \sum_{\ell} |I_\ell| \ \sum_{p=0}^{k-1} \frac{1}{|I_\ell|^2}\ 
\sup_{f:\textup{Lip}_{d_p}(f)\leq 1} \ \left(\int_{I_\ell}
  |g_{k-p,f_p}(x)| dx\right)^2\ (k-p)^{1+\delta}\,. 
$$
Observe that, if $\textup{Lip}_{d_p}(f)\leq 1$, then
$$
\| g_{q,f_p}\|_{L^\infty} \leq \lip_{p+1}(K) \, .
$$
Indeed, 
$$
f_p(x) - \int f_p\ d\mu =\int d\mu_x^{p,0}(z_0^p)\ f(z_0^p) 
- \int d\mu(y) \int d\mu_y^{p,0}(\xi_0^p)\ f(\xi_0^p)= 
$$
$$
\int d\mu_x^{p,0}(z_0^p)\int d\mu(y) \int d\mu_y^{p,0}(\xi_0^p)\ (f(z_0^p)-f(\xi_0^p))\,,
$$
and we use the fact that $\textup{Lip}_{d_p}(f)\leq 1$ and that $\L$ has $L^\infty$-norm equal to one.
This implies that, for any $0<\sigma<2$, 
$$
A_2(k)\leq \Oun\times
$$
$$
\sum_{\ell} \frac{1}{|I_\ell|}
\sum_{p=0}^{k-1} \
 \sup_{f:\textup{Lip}_{d_p}(f)\leq 1} \left(\int_{I_\ell}
   |g_{k-p,f_p}(x)| dx\right)^\sigma  (k-p)^{1+\delta}
 (\lip_{p+1}(K))^{2-\sigma} |I_\ell|^{2-\sigma}\,. 
$$
Using again $h_{|I_\ell}\sim \ell$, Lemma \ref{quiche} and the fact
that $\lip_p(K)\leq \Oun C_p$, we get 
$$
A_2(k)\leq \Oun
\ \sum_{\ell} \frac{|I_\ell|^{1-\sigma}}{\ell^\sigma}\
\sum_{p=0}^{k-1} \ C_p^2 \ \gamma_{k-p}^{\frac{1-\alpha}{3}\sigma}\
(k-p)^{1+\delta}\,. 
$$
Since $\alpha\in[0,\a0[$, there exist $0<\sigma<1$ and $\delta>0$ such that
$$
\sum_q 
\gamma_{q}^{\frac{1-\alpha}{3}\sigma}\ q^{1+\delta}<\infty
$$
then,
$$
\sum_{k=1}^n A_2(k)\leq \Oun \ \sum_{j=1}^n (\lip_j(K))^2\,.
$$
This ends the proof of Lemma \ref{quasidevroye}.
\end{proof}

\subsection{End of the proof}

We now conclude the proof of Theorem \ref{devroye}.
By \eqref{gogo} and \eqref{gogagi}, we have
\begin{eqnarray*}
\expectation\left(K-\expectation K\right)^2 & \leq & 
8 \sum_{i=1}^{n}(\lip_i(K))^2  \\
&& + 2\sum_{i=1}^{n}\int d\mu(x) \sum_{x':T(x')=T(x)} 
\frac{h(x')}{h(T(x))|T'(x')|}\thinspace \Gamma_k^2(x,x')\\
& = & 8\sum_k (\lip_k(K))^2 + 2 \sum_k (S_1(k)+S_2(k))\,.
\end{eqnarray*}

The theorem now follows from Lemmas \ref{crachin} and \ref{quasidevroye}.

\bigskip

\section*{Appendix}

In this appendix we prove the inequalities \eqref{ping} and \eqref{pong}
used in the proof of Lemma \ref{projo}.
We recall that the map $T$ is defined in Section \ref{defs}.

\begin{lemma}\label{ineq1}
There exists a constant $C>0$ such that for any integer $m\geq 1$ and any
pair of points $z$, $\z$ such that for $0\leq j\leq m$, $T^j (z)$ and $T^j (\z)$
belong to the same atom of the Markov partition. Then one has
$$
\frac{|z-\z|}{z}\leq C \frac{|T^m (z)-T^m (\z)|}{T^m z}\cdot
$$
\end{lemma}

\begin{proof}
We start by proving the following inequality:
\begin{equation}\label{BSD}
|T^{m'} (z)| \geq C_0 \left(\frac{T^m (z)}{z}\right)^{1+\alpha}
\end{equation}
where $C_0>0$ is independent of $m$ and $z$.

There are two cases.

If $z\geq1/2$, the inequality is true provided that $C_0\leq 2^{-(1+\alpha)}$. 

Now consider the case $z< 1/2$. 
We define an integer $q\leq m$ as follows.
If $T^j (z) \leq 1/2$ for $j=0,1,\ldots,m-1$ then we take $q=m$.
Otherwise $q$ is the smallest integer sucht that $T^q (z)\geq  1/2$.
Since $z<1/2$, there is an integer $\ell\geq 1$ such that $z\in I_\ell$.
Moreover $T^q$ is a diffeomorphism from $I_\ell$ to $I_{\ell-q}$.
From the distortion lemma, see {\em e.g.} \cite[Proposition 2.3]{hu}, we get
$$
|T^{q'} (z)| \geq C_1 \left(\frac{I_{\ell-q}}{I_\ell}\right)^{1+\alpha}
$$
where $C_1>0$ is independent of $q$ and $z$. 
From \eqref{AMP} it follows that 
$$
|T^{q'} (z)| \geq C_2 \left(\frac{T^q (z)}{z}\right)^{1+\alpha}
$$
where $C_2>0$ is independent of $q$ and $z$. 
If $q=m$ then \eqref{BSD} is proved with $C_0=\min(C_2,2^{-(1+\alpha)})$. If $q<m$ then
we observe that
$$
|T^{m'} (z)| = |T^{(m-q)'} (T^q (z))| |T^{q'} (z)| \geq C_2 \left(\frac{T^q (z)}{z}\right)^{1+\alpha}
$$
because $|T^{(m-q)'}|\geq 1$.
Since $T^q (z) \geq 1/2$, we obtain
$$
|T^{m'} (z)| \geq \frac{C_2}{2^{1+\alpha}} \frac{1}{z^{1+\alpha}} \geq  
\frac{C_2}{2^{1+\alpha}} \left(\frac{y}{z}\right)^{1+\alpha}\cdot
$$
This finishes the proof of inequality \eqref{BSD}.

To prove the lemma, we first observe that if $T^m (z)\leq z$ then
$$
|z-\z|\leq |T^m (z) - T^m (\z)| \leq \frac{z}{T^m (z)}\thinspace |T^m (z) - T^m (\z)| 
$$
because the modulus of the $T'$ is larger than or equal to one.
The remaining case is when $T^m (z) > z$. We observe that
$$
|T^m (z) - T^m (\z) | = \left| \int_z^{\z} T^{m'} (\xi) d\xi \right| =\int_z^{\z} |T^{m'} (\xi)| d\xi \geq 
\tilde{C} \left(\frac{T^m (z)}{z}\right)^{1+\alpha}
$$
where we used again the distortion estimates (\cite[Proposition 2.3]{hu}),
\eqref{AMP}, the monotonicity of $T^m$, and where $\tilde{C}>0$ is
independent $m,z,\z$. This immediately implies
$$
\frac{|z-\z|}{z}\leq \frac{1}{\tilde{C}} \left(\frac{z}{T^m (z)}\right)^\alpha \frac{|T^m (z)-T^m (\z)|}{T^m z}\cdot
$$
The Lemma is proved.
\end{proof}

\begin{lemma}\label{ineq2}
There exists a constant $C>0$ such that for any integer $m\geq 1$ and any
pair of points $z$, $\z$ such that for $0\leq j\leq m$, $T^j (z)$ and $T^j (\z)$
belong to the same atom of the Markov partition. Then one has
$$
|z-\z|\leq \frac{C}{(m+1)^{1/\alpha}} \frac{|T^m (z)-T^m (\z)|}{T^m (z)}\cdot
$$
\end{lemma}

\begin{proof}
Observe that if $T^m (z) <\frac{1}{(m+1)^{1/\alpha}}$ then the estimate follows at once
from the fact that the modulus of the derivative of $T$ is larger than or equal to one.
So we now assume that $T^m (z) \geq \frac{1}{(m+1)^{1/\alpha}}$. 
Let $\ell\geq 0$ be the integer such that $T^m z\in I_\ell$. There exists a unique $z_*$ in $I_{\ell+m}$
such that $T^m (z_*) = T^m (z)$. 
Since $z_*$ is the closest  $T^{-m}$-preimage of $T^m(z)$ to the neutral fixed point $0$, one can
easily show that there is a constant $c>0$ such that, for any $m$ and $z$,
one has $|T^{m'} (z)| \geq c |T^{m'} (z_*)|$.

As in the proof of the previous lemma, we use the distortion estimates (\cite[Proposition 2.3]{hu}) and \eqref{AMP}
to obtain
$$
|T^{m'} (z)| \geq c' \frac{(\ell+m)^{1+\frac{1}{\alpha}}}{\ell^{1+\frac{1}{\alpha}}}
$$
From the distortion estimates (\cite[Proposition 2.3]{hu}) we get, using that $T^m (z)\in I_\ell$,
$$
|T^m (z) - T^m (\z) | \geq  \Oun |T^{m'} (z)|\thinspace |z - \z | \geq
\Oun |T^m (z)| \thinspace \frac{(\ell+m)^{1+\frac{1}{\alpha}}}{\ell} \thinspace |z - \z |.
$$
This can be rewritten as
$$
|z - \z | \leq \Oun \frac{\ell}{(\ell+m)^{1+\frac{1}{\alpha}}} \frac{|T^m (z) - T^m (\z) |}{|T^{m'} z|}\leq 
\Oun \frac{1}{(1+m)^{\frac{1}{\alpha}}} \frac{|T^m (z) - T^m (\z) |}{|T^{m'} (z)|}\cdot
$$
The proof of the lemma is complete.
\end{proof}


\end{document}